\documentclass[submission]{FPSAC2024arxiv}

\newtheorem{thm}{Theorem}
\newtheorem{lem}{Lemma}
\newtheorem{definition}{Definition}
\newtheorem{proposition}{Proposition}

\usepackage{tikz-cd,stackengine,adjustbox}

\newtheorem{innercustomthm}{Theorem}
\newenvironment{customthm}[1]
  {\renewcommand\theinnercustomthm{#1}\innercustomthm}
  {\endinnercustomthm}

\DeclareMathOperator{\grad}{grad}

\DeclareMathOperator{\id}{id}

\DeclareMathOperator{\gr}{gr}

\DeclareMathOperator{\initial}{in}

\DeclareMathOperator{\MultiProj}{MultiProj}

\DeclareMathOperator{\conv}{conv}

\newcommand{\la}{\lambda}
\newcommand{\om}{\omega}
\newcommand{\cU}{\mathcal U}
\newcommand{\cL}{\mathcal L}
\newcommand{\cJ}{\mathcal J}
\newcommand{\cO}{\mathcal O}

\newcommand{\cQ}{\mathcal Q}

\newcommand{\fsl}{\mathfrak{sl}}

\newcommand{\bC}{\mathbb{C}}
\newcommand{\bP}{\mathbb{P}}
\newcommand{\bR}{\mathbb{R}}
\newcommand{\bZ}{\mathbb{Z}}

\newcommand{\ar}{\rangle}
\newcommand{\bs}{\backslash}

\newcommand\ledot{\mathrel{\ensurestackMath{%
  \stackengine{-.5ex}{\lessdot}{-}{U}{c}{F}{F}{S}}}}

\title[Poset polytopes and pipe dreams]{Poset polytopes and pipe dreams: toric degenerations and beyond}

\author{Ievgen Makedonskyi
\addressmark{1}
\and Igor Makhlin
\addressmark{2}}

\address{\addressmark{1}Beijing Institute of Mathematical Sciences and Applications (BIMSA) \\ \addressmark{2}Technical University of Berlin}

\received{\today}

\abstract{We demonstrate how pipe dreams can be applied to the theory of poset polytopes to produce toric degenerations of flag varieties. Specifically, we present such constructions for marked chain-order polytopes of Dynkin types A and C. These toric degenerations also give rise to further algebraic and geometric objects such as PBW-monomial bases and Newton--Okounkov bodies. We discuss a construction of the former in the type A case and of the latter in type C.}


\usepackage[backend=bibtex]{biblatex}
\begin{document}

\maketitle

\section{Introduction}

Recent decades have seen a wide range of new methods for constructing toric degenerations of flag varieties. These methods commonly proceed by attaching a degeneration to every combinatorial or algebraic object of a certain form. Examples of such objects include adapted decompositions in the Weyl group, certain valuations on the function field and certain birational sequences (see~\cite{FaFL1} for details concerning these results and a partial history of the subject). These correspondences are of great interest for a number of reasons, however, not many explicit constructions are known for the attached objects. This leads to a shortage of concrete recipes that would work in a general situation.

Until recently, the only explicit constructions known to work in the generality of all type A flag varieties were the Gelfand--Tsetlin (GT) degeneration due to~\cite{GL} and the Feigin--Fourier--Littelmann-Vinberg (FFLV) degeneration due to~\cite{FFL2} (as well as slight variations of these two). An important step was made by Fujita in~\cite{Fu} where it is proved that each marked chain-order polytope (MCOP) of the GT poset provides a toric degeneration of a type A flag variety. Each such MCOP $\cQ_O(\la)$ is given by a subset $O$ of the GT poset $P$ and an integral dominant $\fsl_n$-weight $\la$. The GT and FFLV polytopes appear as special cases. General MCOPs were defined by Fang and Fourier in~\cite{FF} and present a far-reaching generalization of the poset polytopes considered by Stanley in~\cite{St}.

Now, it must be noted that the main objects of study in~\cite{Fu} are Newton--Okounkov bodies, toric degenerations are obtained somewhat indirectly via a general result of~\cite{A} relating the two notions. This project was initiated with the goal of finding a more direct approach in terms of explicit initial degenerations similar to the classical construction in~\cite{GL}. Recall that $F$, the variety of complete flags in $\bC^n$, is realized by the Pl\"ucker ideal $I$ in the polynomial ring in Pl\"ucker variables $X_{i_1,\dots,i_k}$. Meanwhile, the toric variety of $\cQ_O(\la)$ is realized by a toric ideal $I_O$ in the polynomial ring in variables $X_J$ labeled by order ideals $J$ in $P$. To obtain a toric initial degeneration of $F$ we may find an isomorphism between the two polynomial rings which would map $I_O$ to an initial ideal of $I$. The key challenge is then to define this isomorphism, the solution is provided by pipe dreams: a combinatorial rule for associating a permutation $w_M$ with every subset $M\subset P$.
\begin{customthm}{0}[cf.\ Theorem~\ref{degenmain}]\label{intro}
Fix $O\subset P$. For every order ideal $J$ one can choose $M_J\subset P$ and $k_J\in\mathbb N$ so that the map $\psi:X_J\mapsto X_{w_{M_J}(1),\dots,w_{M_J}(k_J)}$ is an isomorphism and $\psi(I_O)$ is an initial ideal of $I$. Consequently, the toric variety of $\cQ_O(\la)$ is a flat degeneration of $F$.
\end{customthm}

One reason for the popularity of toric degenerations is that they are always accompanied by a collection of other interesting objects: standard monomial theories, Newton--Okounkov bodies, PBW-monomial bases, etc. All of these can also be obtained from our construction, in particular, in this abstract we explain how PBW-monomial bases are obtained in type A. 
Consider the irreducible $\fsl_n$-representation $V_\la$ with highest-weight vector $v_\la$, let $f_{i,j}$ denote the negative root vectors. Then a basis in $V_\la$ is formed by the vectors $\prod f_{i,j}^{x_{i,j}}v_\la$ with $x$ ranging over the lattice points in $\xi(\cQ_O(\la))$ for a certain unimodular transformation $\xi$ (see Theorem~\ref{basismain}). 

We then discuss an extension of our approach to type C. Every integral dominant $\mathfrak{sp}_{2n}$-weight $\la$ and every subset $O$ of the type C GT poset also define an MCOP $\cQ_O(\la)$. Using a notion of type C pipe dreams (not to be confused with other known symplectic pipe-dream analogs) we state a type C counterpart of Theorem~\ref{intro}, see Theorem~\ref{degenmainC}. A notable new feature of this case is the intermediate degeneration $\widetilde F$ of the symplectic flag variety which happens to be a type A Schubert variety (Theorem~\ref{intermediate}, compare also~\cite{CL}). All of our toric degenerations are obtained as further degenerations of $\widetilde F$.

We also use type C to showcase another aspect of the theory: Newton--Okounkov bodies. Namely, we show how every $\cQ_O(\la)$ can be realized as a Newton--Okounkov body of the symplectic flag variety (Theorem~\ref{NOmain}).
To us this result is of particular interest because the paper~\cite{Fu} explains in detail why its methods do not extend to type C.

Complete proofs and further context for the results in type A can be found in~\cite{M3} (which also discusses an extension to semi-infinite Grassmannians), the paper~\cite{MM} covers type C (and also type B).

\section{Type A}

\subsection{Poset polytopes}

Choose an integer $n\ge 2$ and consider the set of pairs $P=\{(i,j)\}_{1\le i\le j\le n}$. We define a partial order $\prec$ on $P$ by setting $(i,j)\preceq(i',j')$ if and only if $i\le i'$ and $j\le j'$. The poset $(P,\prec)$ is sometimes referred to as the \textit{Gelfand--Tsetlin \textnormal{(or} GT\textnormal) poset}. 
We denote $A=\{(i,i)\}_{i\in[1,n]}\subset P$. Let $\cJ$ be the set of order ideals (lower sets) in $(P,\prec)$. For $k\in[0,n]$ let $\cJ_k\subset\cJ$ consist of $J$ such that $|J\cap A|=k$, i.e.\ $J$ contains $(1,1),\dots,(k,k)$ but not $(k+1,k+1)$.

We now associate a family of polytopes with this poset. Each polytope is determined by a subset of $P$ and a vector in $\bZ_{\ge 0}^{n-1}$. For $k\in[1,n-1]$ we let $\om_k$ denote the $k$th basis vector in $\bZ_{\ge 0}^{n-1}$.
\begin{definition}\label{mcopdef}
Consider a subset $O\subset P$ such that $A\subset O$. For $J\in\cJ$ consider the set \[M_O(J)=(J\cap O)\cup\max\nolimits_{\prec}(J)\] ($\max_\prec$ denotes the subset of $\prec$-maximal elements). Let $x_O(J)\in\bR^P$ denote the indicator vector $\mathbf 1_{M_O(J)}$. The \textit{marked chain-order polytope} (MCOP) $\cQ_O(\om_k)$ is the convex hull of $\{x_O(J)\}_{J\in\cJ_k}$. For $\la=(a_1,\dots,a_{n-1})\in\bZ_{\ge0}^{n-1}$ the MCOP $\cQ_O(\la)$ is the Minkowski sum \[a_1\cQ_O(\om_1)+\dots+a_{n-1}\cQ_O(\om_{n-1})\subset\bR^P.\]
\end{definition}

MCOPs were introduced in~\cite{FF,FFLP} in the generality of arbitrary finite posets. The original definition describes the polytope in terms of linear inequalities. The equivalence of the above approach is proved in~\cite[Subsection 3.5]{FM2}.

The first thing to note is that for $\la=(a_1,\dots,a_n)$ and any $x\in\cQ_O(\la)$ one has $x_{i,i}=a_i+\dots+a_{n-1}$. 
When $O=P$ one has $M_O(J)=J$. It follows that $\cQ_P(\la)$ consists of points $x$ with $x_{i,j}\ge x_{i',j'}$ whenever $(i,j)\preceq(i',j')$. Now, identify $\bZ^{n-1}$ with the lattice of integral $\fsl_n$-weights by letting $\om_k$ be the $k$th fundamental weight. Then $\cQ_P(\la)$ is the GT polytope of~\cite{GT} corresponding to the integral dominant weight $\la$ (i.e.\ $\la\in \bZ_{\ge0}^{n-1}$). If $O=A$, then $M_O(J)$ is the union of $J\cap A$ and the antichain $\max_\prec J$. One can check that $\cQ_A(\la)$ consists of points $x$ with all $x_{i,j}\ge 0$ and $\sum_{(i,j)\in K} x_{i,j}\le a_l+\dots+a_r$ for any chain $K\subset P\bs A$ starting in $(l,l+1)$ and ending in $(r,r+1)$. This is the FFLV polytope (\cite{FFL1}) of $\la$. Other MCOPs can be said to interpolate between these two cases.


Note that $|\cJ_k|={n\choose k}$ and, since $\cQ_O(\om_k)$ is a $0/1$-polytope, it has ${n\choose k}$ lattice points. More generally, a key property of MCOPs is that the number of lattice points in $\cQ_O(\la)$ does not depend on $O$ and, moreover, the polytopes with a given $\la$ are pairwise Ehrhart-equivalent (\cite[Corollary 2.5]{FFLP}). Now, it is well known that the number of lattice points in the GT or FFLV polytope of $\la$ is $\dim V_\la$ where $V_\la$ denotes the irreducible $\fsl_n(\bC)$-representation with highest weight $\la$. This immediately provides the following.
\begin{proposition}\label{intpoints}
For any $O$ and $\la$ we have $|\cQ_O(\la)\cap\bZ^P|=\dim V_\la$.
\end{proposition}

The polytope $\cQ_O(\om_k)$ is normal which means that the associated toric variety is embedded into $\bP(\bC^{\cJ_k})$. It is cut out by the kernel of the homomorphism from $\bC[X_J]_{J\in\cJ_k}$ to $\bC[P]=\bC[z_{i,j}]_{(i,j)\in P}$ mapping $X_J$ to $z^{x_O(J)}=\prod_{(i,j)\in P} z_{i,j}^{x_O(J)_{i,j}}$. For general $\la=(a_1,\dots,a_n)$ the definition of $\cQ_O(\la)$ implies that its normal fan and its toric variety (up to isomorphism) depend only on the set of $i$ for which $a_i>0$. Hence, for regular $\la$ (all $a_i>0$) the toric variety coincides with that of $\cQ_O(\om_1)+\dots+\cQ_O(\om_{n-1})$. The toric variety of a Minkowski sum has a standard multiprojective embedding. Consider the product \[\bP_\cJ=\bP(\bC^{\cJ_1})\times\dots\times\bP(\bC^{\cJ_{n-1}})\] and its multihomogeneous coordinate ring $\bC[\cJ]=\bC[X_J]_{J\in\cJ_1\cup\dots\cup\cJ_{n-1}}$. Let $I_O$ denote the kernel of the homomorphism $X_J\mapsto z^{x_O(J)}$ from $\bC[\cJ]$ to $\bC[P]$.
\begin{proposition}\label{hibiproj}
For regular $\la$ the toric variety of $\cQ_O(\la)$ is isomorphic to the zero set of $I_O$~in~$\bP_\cJ$.
\end{proposition}

\subsection{Pipe dreams}

Consider the permutation group $\mathcal S_n$ and for $(i,j)\in P$ let $s_{i,j}$ denote the transposition $(i,j)\in\mathcal S_n$. In particular, $s_{i,i}$ is always the identity.
\begin{definition}
For any subset $M\subset P$ let $w_M\in\mathcal S_n$ denote the product of all $s_{i,j}$ with $(i,j)\in M$ ordered first by $i$ increasing from left to right and then by $j$ increasing from left to right.
\end{definition}

Note that $w_M$ is determined by $M\backslash A$ but it is convenient for us to consider subsets of $P$ rather than $P\bs A$. The term \textit{pipe dream} is due to~\cite{KnM} and refers to a certain diagrammatic interpretation of this correspondence between subsets of $P$ and permutations. The poset $P$ can be visualized as a triangle as shown in~\eqref{pipedreamfig} for $n=4$. In these terms the pipe dream corresponding to $M$ consists of $n$ polygonal curves or \textit{pipes} described as follows. The $i$th pipe enters the element $(i,n)$ from the bottom-right, continues in this direction until it reaches an element of $M\cup A$, after which it turns left and continues going to the bottom-left until it reaches an element of $M$, after which it turns right and again continues to the top-right until it reaches an element of $M\cup A$, etc. The last element passed by the pipe will then be $(1,w_M(i))$. 

The pipe dream of the set $M=\{(1,1),(2,2),(1,2),(2,3),(1,4)\}$ is shown below, here each pipe is shown in its own colour. Indeed, $s_{1,1}s_{1,2}s_{1,4}s_{2,2}s_{2,3}=(4,3,1,2)$.
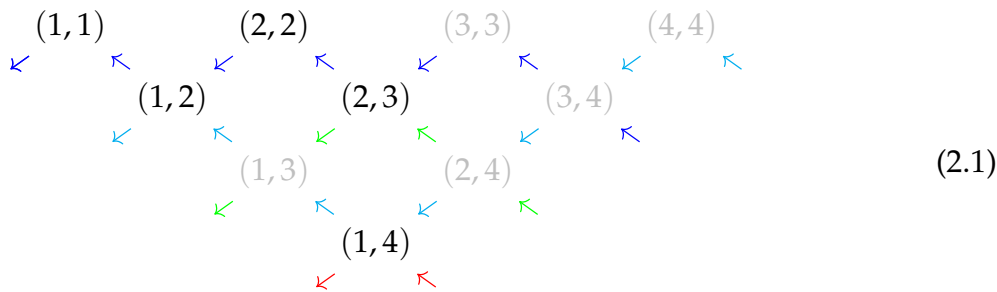
\begin{equation}\label{pipedreamfig}
\begin{tikzcd}[row sep=tiny,column sep=0]
&(1,1)\ar[blue]{ld}\ar[blue]{ld}&&(2,2)\ar[blue]{ld}&&\color{lightgray}{(3,3)}\ar[blue]{ld}&&\color{lightgray}{(4,4)}\ar[cyan]{ld}\\
\phantom{(1,1)}&&(1,2)\ar[blue]{lu}\ar[cyan]{ld}&&(2,3)\ar[green]{ld}\ar[blue]{lu}&&\color{lightgray}{(3,4)}\ar[blue]{lu}\ar[cyan]{ld}&&\phantom{(1,1)}\ar[cyan]{lu}\\
&\phantom{(1,1)}&&\color{lightgray}{(1,3)}\ar[green]{ld}\ar[cyan]{lu}&&\color{lightgray}{(2,4)}\ar[green]{lu}\ar[cyan]{ld}&&\phantom{(1,1)}\ar[blue]{lu}\\
&&\phantom{(1,1)}&&(1,4)\arrow[red]{ld}\ar[cyan]{lu}&&\phantom{(1,1)}\ar[green]{lu}\\
&&&\phantom{(1,1)}&&\phantom{(1,1)}\arrow[red]{lu}
\end{tikzcd}
\end{equation}

We will use a ``twisted'' version of the correspondence depending on $O\subset P$. For $M\subset P$ we denote $w^O_M=w_O^{-1}w_M$. Diagrammatically, $w^O_M(i)=j$ if the $i$th pipe of the pipe dream of $M$ ends in the same element as the $j$th pipe of the pipe dream of $O$.

\subsection{Toric degenerations}

For a polynomial ring $\bC[x_a]_{a\in A}$ a \textit{monomial order} $<$ on $\bC[x_a]_{a\in A}$ is a partial order on the set of monomials that is multiplicative ($M_1<M_2$ if and only if $M_1x_a<M_2x_a$) and weak (incomparability is an equivalence relation). For such an order and a polynomial $p\in \bC[x_a]_{a\in A}$ its \textit{initial part} $\initial_< p$ is equal to the sum of those monomials occurring in $p$ which are maximal with respect to $<$, taken with the same coefficients as in $p$. For a subspace $U\subset\bC[x_a]_{a\in A}$ its \textit{initial subspace} $\initial_< U$ is the linear span of all $\initial_< p$ with $p\in U$. The initial subspace of an ideal is an ideal (the \textit{initial ideal}), the initial subspace of a subalgebra is a subalgebra (the \textit{initial subalgebra}). One key property of initial ideals and initial subalgebras is that they define flat degenerations, we explain this phenomenon in the context of flag varieties.

For $n\ge 2$ let $F$ be the variety of complete flags in $\bC^n$. The Pl\"ucker embedding realizes $F$ as a subvariety in \[\bP=\bP(\wedge^1\bC^n)\times\dots\times\bP(\wedge^{n-1}\bC^n).\] The multihomogeneous coordinate ring of $\bP$ is $S=\bC[X_{i_1,\dots,i_k}]_{k\in[1,n-1],1\le i_1<\dots<i_k\le n}$ and $F$ is cut out in $\bP$ by the \textit{Pl\"ucker ideal} $I\subset S$ which can be defined as follows. Consider the $n\times n$ matrix $Z$ with $Z_{i,j}=z_{i,j}$ if $i\le j$ and $Z_{i,j}=0$ otherwise. Denote by $D_{i_1,\dots,i_k}$ the minor of $Z$ spanned by rows $1,\dots,k$ and columns $i_1,\dots,i_k$. Then $I$ is the kernel of the homomorphism $\varphi:X_{i_1,\dots,i_k}\mapsto D_{i_1,\dots,i_k}$ from $S$ to $\bC[P]$. One can also equip $S$ with a $\bZ^{n-1}$-grading $\grad$ with $\grad X_{i_1,\dots,i_k}=\om_k$ and characterize $F$ as $\MultiProj{S/I}$ with respect to the induced $\bZ^{n-1}$-grading. The following fact is essentially classical, for the context of partial monomial orders see~\cite{KNN} (where an algebraic wording is given).
\begin{proposition}
For a monomial order $<$ on $S$ the scheme $\MultiProj{S/\initial_< I}$ (i.e.\ the zero set of $\initial_< I$ in $\bP$ if the scheme is reduced) is a flat degeneration of $F$: there exists a flat family $\mathcal F\to\mathbb A^1$ with fiber over 0 isomorphic to $\MultiProj{S/\initial_< I}$ and all other fibers isomorphic to $F$.
\end{proposition}

Now fix $O\subset P$ containing $A$. For $J\in\cJ$ denote $w^O_{M_O(J)}=w^J$. The key ingredient of our first main result is a homomorphism $\psi:\bC[\cJ]\to S$. To define $\psi$ for $J\in\cJ_k$ we set \[\psi(X_J)=X_{w^J(1),\dots,w^J(k)}\] where we use the convention $X_{i_1,\dots,i_k}=(-1)^\sigma X_{i_{\sigma(1)},\dots,i_{\sigma(k)}}$ for $\sigma\in\mathcal S_k$. The map $\psi$ encodes a correspondence $J\mapsto (w^J(1),\dots,w^J(k))$ between order ideals and tuples. When $O=P$ the tuples obtained in this way are precisely the increasing tuples. When $O=A$ one obtains the \textit{PBW tuples} defined in~\cite{Fe}. In general, every subset of $[1,n]$ is represented by exactly one of the obtained tuples, this means that $\psi$ is an isomorphism.

\begin{thm}\label{degenmain}
The map $\psi$ is an isomorphism and there exists a monomial order $<$ on $S$ such that $\psi(I_O)=\initial_< I$. In particular, for regular $\la$ the toric variety of $\cQ_O(\la)$ is a flat degeneration of $F$.
\end{thm}
\begin{proof}[Sketch of proof]
It can be checked that $w^J(i)\ge i$ for $J\in\cJ_k$ and $i\in[1,k]$. Furthermore, there exists a unimodular transformation $\xi\in SL(\bZ^P)$ such that for any $J\in\cJ_k$ one has \[\xi(x_O(J))=\mathbf 1_{\{(1,w^J(1)),\dots,(k,w^J(k))\}}.\] This means that $I_O$ is the kernel of the homomorphism $X_J\mapsto z_{1,w^J(1)}\dots z_{k,w^J(k)}$.

Using pipe dreams one can define a lexicographic monomial order $\ll$ on $\bC[P]$ so that 
\[
\initial_\ll D_{w^J(1),\dots,w^J(k)}=z_{1,w^J(1)}\dots z_{k,w^J(k)}
\]
for $J\in\cJ_k$. Since $\xi$ is bijective, the right-hand sides are distinct monomials for distinct $J$, hence the sets $\{w^J(1),\dots,w^J(k)\}$ are also pairwise distinct. This provides the isomorphism claim. Note that $\psi(I_O)$ is the kernel of $X_{i_1,\dots,i_k}\mapsto\initial_\ll D_{i_1,\dots,i_k}$. Moreover, the $\initial_\ll D_{i_1,\dots,i_k}$ generate the initial subalgebra $\initial_\ll\varphi(S)$ (i.e.\ the determinants form a \textit{sagbi basis}). By general properties of initial degenerations, $\ll$ can now be pulled back to a monomial order $<$ on $S$ with the desired property.
\end{proof}

In fact, we could define $\psi$ using the ``untwisted'' permutation $w_{M(J)}$ and Theorem~\ref{degenmain} would still hold since $I$ is invariant under $\mathcal S_n$. However, we consider $w^J$ the natural choice because of the property $w^J(i)\ge i$, $i\in[1,k]$ which is also crucial in the next~subsection.

\subsection{PBW-monomial bases}

The map $\xi$ considered in the proof sketch of Theorem~\ref{degenmain} maps $\cQ_O(\om_k)$ to $\Pi_O(\om_k)$: the convex hull of all $\mathbf 1_{\{(1,w^J(1)),\dots,(k,w^J(k))\}}$ with $J\in\cJ_k$. For $\la=(a_1,\dots,a_n)$ let $\Pi_O(\la)$ denote the image $\xi(\cQ_O(\la))$. It equals the Minkowski sum $a_1\Pi_O(\om_1)+\dots+a_{n-1}\Pi_O(\om_{n-1})$. 

Next, let us recall some standard Lie-theoretic notation. We have identified $\bZ^{n-1}$ with the lattice of integral $\fsl_n$-weights, let $\alpha_1,\dots,\alpha_{n-1}\in\bZ^{n-1}$ denote the simple roots, i.e.\ $\alpha_i=2\om_i-\om_{i-1}-\om_{i+1}$ where $\om_0=\om_n=0$. The positive roots are then $\alpha_{i,j}=\alpha_i+\dots+\alpha_{j-1}$ with $1\le i<j\le n$. Let $f_{i,j}\in\fsl_n(\bC)$ denote the negative root vector of weight $-\alpha_{i,j}$. For $x\in\bZ_{\ge 0}^P$ we write $f^x$ to denote the PBW monomial $\prod_{(i,j)\in P\bs A} f_{i,j}^{x_{i,j}}$ in $\cU(\fsl_n(\bC))$ ordered first by $i$ increasing from left to right and then by $j$ increasing from left to right. Finally, let $v_\la$ denote a chosen highest-weight vector in $V_\la$. Another of our main results is as follows.
\begin{thm}\label{basismain}
The vectors $f^xv_\la$ with $x\in\Pi_O(\la)\cap\bZ^P$ form a basis in $V_\la$.  
\end{thm}

When $O=A$ the transformation $\xi$ is almost the identity: one has $\xi(x)_{i,j}=x_{i,j}$ for all $i<j$ so that $f^{\xi(x)}=f^x$. Since $O_A(\la)$ is the FFLV polytope, one sees that the obtained basis is the FFLV basis of~\cite{FFL1}. For $O=P$ the corresponding basis is also known, see, for instance,~\cite{R,M1}. Since in this case the tuples $(w^J(1),\dots,w^J(k))$ are increasing, the definition of $\Pi_P(\la)$ is particularly simple. The observation that such a polytope $\Pi_P(\la)$ is unimodularly equivalent to the GT polytope $\cQ_O(\la)$ is due to~\cite{KM}.


\section{Type C}

\subsection{Type C poset polytopes}

For $n\ge 2$ consider the totally ordered set $(N,\lessdot)=\{1\lessdot\dots\lessdot n\lessdot -n\lessdot\dots\lessdot -1\}$. 

\begin{definition}
The \textit{type C GT poset} $(P,\prec)$ consist of pairs of integers $(i,j)$ such that $i\in[1,n]$ and $j\in[i,n]\cup[-n,-i]$. The order relation is given by $(i_1,j_1)\preceq(i_2,j_2)$ if and only if $i_1\le i_2$ and $j_1\ledot j_2$.
\end{definition}

$(P,<)$ has length $2n$, below is its Hasse diagram for $n=2$.
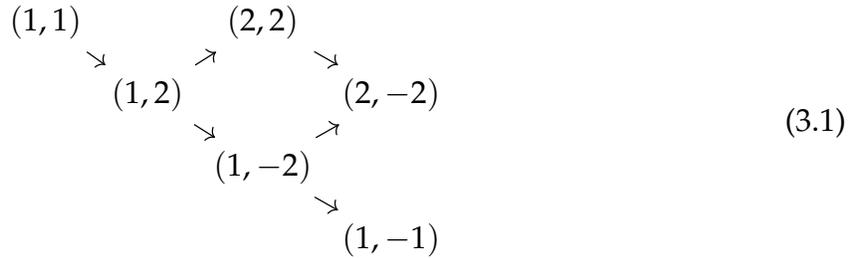
\begin{equation}
\begin{tikzcd}[row sep=tiny,column sep=0]\label{hasseC}
(1,1)\arrow[rd]&&(2,2)\arrow[rd]\\
&(1,2)\arrow[rd]\arrow[ru]&&(2,-2)\\
&&(1,-2)\arrow[ru]\arrow[rd]\\
&&&(1,-1)
\end{tikzcd}
\end{equation}

We use notation similar to type A. Let $A\subset P$ be the set of all $(i,i)$. Let $\cJ$ denote the set of order ideals in $(P,\prec)$. For $k\in[1,n]$ let $\cJ_k$ consist of $J$ such that $|J\cap A|=k$. We also consider the lattice $\bZ^n$ with $\om_k$ denoting the $k$th basis vector. The definition of MCOPs is almost identical.
\begin{definition}\label{mcopdefC}
Consider a subset $O\subset P$ such that $A\subset O$. For $J\in\cJ$ consider the set \[M_O(J)=(J\cap O)\cup\max\nolimits_{\prec}(J).\] Let $x_O(J)\in\bR^P$ denote the indicator vector $\mathbf 1_{M_O(J)}$. The MCOP $\cQ_O(\om_k)$ is the convex hull of $\{x_O(J)\}_{J\in\cJ_k}$. For $\la=(a_1,\dots,a_n)\in\bZ_{\ge0}^n$ the MCOP $\cQ_O(\la)$ is the Minkowski sum \[a_1\cQ_O(\om_1)+\dots+a_n\cQ_O(\om_n)\subset\bR^P.\]
\end{definition}

We identify $\bZ^n$ with the lattice of integral $\mathfrak{sp}_{2n}$-weights with $\om_k$ being the $k$th fundamental weight. Then for an integral dominant weight $\la\in\bZ_{\ge0}^n$ one sees that $\cQ_P(\la)$ is the type C Gelfand-Tsetlin polytope defined in~\cite{BZ} while $\cQ_A(\la)$ is the type C FFLV polytope defined in~\cite{FFL2}. Both of these polytopes are known to parametrize bases in $V_\la$, the irreducible $\mathfrak{sp}_{2n}(\bC)$-representation with highest weight $\la$. This again provides
\begin{proposition}\label{intpointsC}
For any $O$ and $\la$ we have $|\cQ_O(\la)\cap\bZ^P|=\dim V_\la$.
\end{proposition}

We also have multiprojective embeddings for toric varieties. Consider the product \[\bP_\cJ=\bP(\bC^{\cJ_1})\times\dots\times\bP(\bC^{\cJ_n})\] and its multihomogeneous coordinate ring $\bC[\cJ]=\bC[X_J]_{J\in\cJ_1\cup\dots\cup\cJ_n}$. Let $I_O$ denote the kernel of the homomorphism $\varphi_O:X_J\mapsto z^{x_O(J)}$ from $\bC[\cJ]$ to $\bC[P]=\bC[z_{i,j}]_{(i,j)\in P}$.
\begin{proposition}\label{hibiprojC}
For regular $\la$ the toric variety of $\cQ_O(\la)$ is isomorphic to the zero set of $I_O$~in~$\bP_\cJ$.
\end{proposition}

\subsection{Type C pipe dreams}

Let $S_N$ denote the group of all permutations of the set $N$. For $(i,j)\in P$ let $s_{i,j}\in S_N$ denote the transposition which exchanges $i$ and $j$ and fixes all other elements ($s_{i,i}=\id$). 

\begin{definition}
For $M\subset P$ let $w_M\in S_N$ be the product of all $s_{i,j}$ with $(i,j)\in M$ ordered first by $i$ increasing from left to right and then by $j$ increasing with respect to $\lessdot$ from left to right.
\end{definition}

In this case the pipe dream consists of $2n$ pipes enumerated by $N$. In terms of the visualization in~\eqref{hasseC} the $i$th pipe with $i\in[1,n]$ enters the element $(i,-i)$ from the \textbf{bottom}-right and turns at elements of $M\cup A$ while the $i$th pipe with $i\in[-n,-1]$ enters the element $(i,-i)$ from the \textbf{top}-right and then also turns at elements of $M\cup A$. 

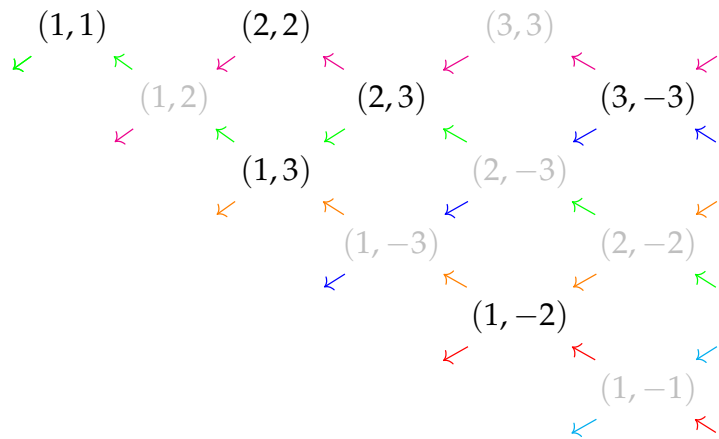
\begin{center}
\begin{tikzcd}[row sep=tiny,column sep=0]
&(1,1)\ar[green]{ld}\ar[green]{ld}&&(2,2)\ar[magenta]{ld}&&\color{lightgray}{(3,3)}\ar[magenta]{ld}&&\phantom{(1,1)}\ar[magenta]{ld}\\
\phantom{(1,1)}&&\color{lightgray}(1,2)\ar[green]{lu}\ar[magenta]{ld}&&(2,3)\ar[green]{ld}\ar[magenta]{lu}&&{(3,-3)}\ar[blue]{ld}\ar[magenta]{lu}&&\phantom{(1,1)}\\
&\phantom{(1,1)}&&(1,3)\ar[green]{lu}\ar[orange]{ld}&&\color{lightgray}{(2,-3)}\ar[green]{lu}\ar[blue]{ld}&&\phantom{(1,1)}\ar[blue]{lu}\ar[orange]{ld}\\
&&\phantom{(1,1)}&&\color{lightgray}(1,-3)\ar[blue]{ld}\ar[orange]{lu}&&\color{lightgray}(2,-2)\ar[green]{lu}\ar[orange]{ld}&&\phantom{(1,1)}\\
&&&\phantom{(1,1)}&&(1,-2)\ar[red]{ld}\ar[orange]{lu}&&\phantom{(1,1)}\ar[green]{lu}\ar[cyan]{ld}\\
&&&&\phantom{(1,1)}&&\color{lightgray}(1,-1)\ar[red]{lu}\ar[cyan]{ld}\\
&&&&&\phantom{(1,1)}&&\phantom{(1,1)}\ar[red]{lu}
\end{tikzcd}
\end{center}

The pipe dream for the set $M=\{(1,1),(1,3),(1,-2),(2,2),(2,3),(3,-3)\}$ is shown above with each pipe in its own colour. One obtains \[w_M(1,2,3,-3,-2,-1)=(-2,1,-3,2,3,-1)\] which agrees with $w_M=s_{1,1}s_{1,3}s_{1,-2}s_{2,2}s_{2,3}s_{3,-3}$.

In fact, pipe dreams for type $\mathrm{C}_n$ can be viewed as a special case of pipe dreams for type $\mathrm{A}_{2n-1}$, i.e.\ for $\fsl_{2n}$. Indeed, one may identify the type $\mathrm{C}_n$ GT poset with the ``left half'' of the type $\mathrm{A}_{2n-1}$ GT poset. Then the $2n$ pipes in the type C pipe dream of $M$ will just be end parts of the $2n$ pipes in the type A pipe dream of the same set $M$.


We again introduce a ``twisted'' version of the correspondence determined by the choice of $O\subset P$: set $w^O_M=w_O^{-1}w_M$.

\subsection{The intermediate Schubert degeneration}

We now construct a degeneration $\widetilde F$ of the symplectic flag variety which will be used as an intermediate step: toric degenerations will be obtained as further degenerations of $\widetilde F$.

\begin{definition}
A tuple $(i_1,\dots,i_k)$ of elements of $N$ is \textit{admissible} if for every $l\in[1,n]$ the number of elements with $|i_j|\le l$ does not exceed $l$. Let $\Theta$ denote the set of all admissible tuples of the form $(i_1\lessdot\dots\lessdot i_k)$ and $\Theta_k\subset\Theta$ denote the subset of $k$-tuples.
\end{definition}

Consider the space $\bC^N\simeq\bC^{2n}$ with basis $\{e_i\}_{i\in N}$. One has a standard embedding $V_{\om_k}\subset\wedge^k\bC^N$. Let $\{X_{i_1,\dots,i_k}\}_{i_1\lessdot\dots\lessdot i_k}$ be the basis in $(\wedge^k\bC^N)^*$ dual to $\{e_{i_1}\wedge\dots\wedge e_{i_k}\}_{i_1\lessdot\dots\lessdot i_k}$. It is known (\cite{DC}) that the set $\{X_{i_1,\dots,i_k}\}_{(i_1,\dots,i_k)\in\Theta_k}$ projects to a basis in $V_{\om_k}^*$.
This allows us to identify $V_{\om_k}^*$ with $\bC^{\Theta_k}$. The multihomogeneous coordinate ring of  \[\bP=\bP(V_{\om_1})\times\dots\times\bP(V_{\om_n})\] is then identified with $\bC[\Theta]=\bC[X_{i_1,\dots,i_k}]_{(i_1,\dots,i_k)\in\Theta}$. The Pl\"ucker embedding of the complete symplectic flag variety $F\hookrightarrow\bP$ is defined by the symplectic Pl\"ucker ideal $I\subset\bC[\Theta]$.

Next, consider the variety $F_\mathrm{A}$ of type A partial flags in $\bC^N$ of signature $(1,\dots,n)$. It is embedded into \[\bP_\mathrm{A}=\bP(\wedge^1\bC^N)\times\dots\times\bP(\wedge^n\bC^N)\] where it is cut out by the Pl\"ucker ideal $I_\mathrm{A}$ in the multihomogeneous coordinate ring $S=\bC[X_{i_1,\dots,i_k}]_{k\in[1,n],\{i_1\lessdot\dots\lessdot i_k\}\subset N}$. Consider the surjection $\pi:S\to\bC[\Theta]$ mapping all $X_{i_1,\dots,i_k}\notin\bC[\Theta]$ to 0 and fixing $\bC[\Theta]$. Set $\widetilde I=\pi(I_\mathrm{A})$.
\begin{thm}\label{intermediate}
There exists a monomial order $\widetilde<$ on $\bC[\Theta]$ such that $\initial_{\widetilde <} I=\widetilde I$.
\end{thm}

This means that $\widetilde F$, the zero set of $\widetilde I$ in $\bP$, is a flat degeneration of $F$. Also, importantly to us, every initial ideal of $\widetilde I$ is an initial ideal of $I$. The advantage of working with $\widetilde I$ instead of degenerating $I$ directly is that the former is more simply expressed as a homomorphism kernel, allowing us to use the technique of sagbi degenerations. Consider $Z$, the $n\times 2n$ matrix with rows indexed by $[1,n]$ and columns indexed by $N$ such that $Z_{i,j}=z_{i,j}$ if $(i,j)\in P$ and $Z_{i,j}=0$ otherwise. Denote by $D_{i_1,\dots,i_k}$ the minor of $Z$ spanned by rows $1,\dots,k$ and columns $i_1,\dots,i_k$.
\begin{proposition}\label{intermediatekernel}
$\widetilde I$ is the kernel of the homomorphism $X_{i_1,\dots,i_k}\mapsto D_{i_1,\dots,i_k}$ from $\bC[\Theta]$ to $\bC[P]$.    
\end{proposition}

A noteworthy property of $\widetilde F$ is that it is a type A Schubert variety in $F_\mathrm{A}$. Indeed, $\pi^{-1}(\widetilde I)$ cuts out $\widetilde F$ in $\bP_\mathrm{A}$ and $\pi^{-1}(\widetilde I)$ is generated by $I_\mathrm{A}$ and all $X_{i_1,\dots,i_k}\notin\bC[\Theta]$. Consider the alternative order $-1\lessdot'1\lessdot'\dots\lessdot'-n\lessdot n$ on $N$. One sees that $X_{i_1,\dots,i_k}\notin\bC[\Theta]$ if and only if $(i_1,\dots,i_k)$ has a reordering $(j_1\lessdot'\dots\lessdot'j_k)$ such that $j_l\lessdot' -l$ for some $l$. Now one sees that $\pi^{-1}(\widetilde I)$ is indeed the defining ideal of a Schubert variety in $F_\mathrm{A}$. Namely, the Schubert variety corresponding to the Borel subgroup in $SL(\bC^N)$ given by the ordering $\lessdot'$ and the torus-fixed point $y\in F_\mathrm{A}$ with all multihomogeneous coordinates zero except for $y_{-1,\dots,-k}$ with $k\in[1,n]$.

\subsection{Toric degenerations}

Fix $O\subset P$ containing $A$. We can now realize the toric variety of $\cQ_O(\la)$ as a degeneration of $F$ by identifying $I_O$ with an initial ideal of $\widetilde I$. This is again done via an isomorphism between $\bC[\cJ]$ and $\bC[\Theta]$. For $J\in\cJ$ denote $w^O_{M_O(J)}=w^J$.
\begin{lem}
For every $J\in\cJ_k$ and $i\in[1,k]$ one has $|w^J(i)|\ge i$. In particular, $(w^J(1),\dots,w^J(k))$ is admissible. 
\end{lem}

The lemma lets us define a homomorphism $\psi:\bC[\cJ]\to\bC[\Theta]$, for $J\in\cJ_k$ we set\[\psi(X_J)=X_{w^J(1),\dots,w^J(k)}.\] 
\begin{thm}\label{degenmainC}
The map $\psi$ is an isomorphism and for a certain (explicitly defined) monomial order $<$ on $S$ one has $\psi(I_O)=\initial_< \widetilde I$. In particular, for regular $\la$ the toric variety of $\cQ_O(\la)$ is a flat degeneration of $\widetilde F$ and, subsequently, of $F$.
\end{thm}

\subsection{Newton--Okounkov bodies}

Following~\cite{Ka} we associate a Newton--Okounkov body of $F$ with a line bundle $\cL$, a global section $\tau$ of $\cL$ and a valuation $\nu$ on the function field $\bC(F)$. We choose an integral dominant $\la=(a_1,\dots,a_n)$ and let $\cL$ be the $Sp_{2n}(\bC)$-equivariant line bundle on $F$ associated with the weight $\la$. In terms of the multiprojective embedding $F\subset\bP$ this is the restriction of $\cO(a_1,\dots,a_n)$ to $F$. Consider the $\bZ^N$-grading on $\bC[\Theta]$ given by $\grad X_{i_1,\dots,i_k}=\om_k$ and the induced grading on the Pl\"ucker algebra $R=C[\Theta]/I$. Then $H^0(F,\cL)$ is identified with the homogeneous component of grading $\la$ in $R$. No we choose $\tau\in H^0(F,\cL)$ as the image of $\prod_k X_{1,\dots,k}^{a_k}$ in $R$.

To define the valuation $\nu$ we first  define a valuation on $R$. Theorems~\ref{intermediate} and~\ref{degenmainC} provide a monomial order $<_0$ on $\bC[\Theta]$ such that $\initial_{<_0} I=\psi(I_O)$. The proofs of the theorems show that $<_0$ arises from a total monomial order $\ll$ on $\bC[P]$ in the following sense. Consider the homomorphism $\rho:\bC[\Theta]\to\bC[P]$ mapping the variable $\psi(X_J)$ to $z^{x_O(J)}$, note that $\rho=\varphi_O\psi^{-1}$ and $\ker\rho=\psi(I_O)$. Then for two monomials one has $M_1<_0 M_2$ if and only if $\rho(M_1)\ll\rho(M_2)$. The monomial order $\ll$ corresponds to a semigroup order on $\bZ_{\ge 0}^P$ which we also denote by $\ll$. We have a $(\bZ_{\ge 0}^P,\ll)$-filtration on $R$ with component $R_x$, $x\in\bZ_{\ge 0}^P$ spanned by the images of monomials $M\in\bC[\Theta]$ such that $\rho(M)\ll z^x$. By general properties of initial degenerations we then have $\gr R\simeq \bC[\Theta]/\initial_{<_0} I$ (which is the toric ring of $\cQ_O(\la)$). For nonzero $p\in R$ we now define $\nu(p)$ as the $\ll$-minimal $x$ such that $p\in R_x$: by definition, such a map is a valuation if and only if $\gr R$ is an integral domain. One sees that $\nu$ maps the (image in $R$ of) $\psi(X_J)$ to $x_O(J)$ and, consequently, \[\nu(H^0(F,\cL)\bs\{0\})=\cQ_O(\la)\cap\bZ^P.\] Since $\bC(F)$ consists of fractions $p/q$ where $\grad$-homogeneous $p,q\in R$ satisfy $\grad p=\grad q$, we can now extend the valuation to $\bC(F)$ by $\nu(p/q)=\nu(p)-\nu(q)$.

\begin{definition}
The \textit{Newton--Okounkov body} of $F$ defined by $\cL$, $\tau$ and $\nu$ is the convex hull closure \[\Delta=\overline{\conv\left\{\left.\frac{\nu(\sigma/\tau^m)}m\right|m\in\bZ_{>0},\sigma\in H^0(F,\cL^{\otimes m})\bs\{0\}\right\}}\subset\bR^P.\]
\end{definition}


For every $k\in[1,n]$ we have a unique $J\in\cJ_k$ such that $w^J(\{1,\dots,k\})=\{1,\dots,k\}$, denote $x_k=x_O(J)$. For $\la=(a_1,\dots,a_n)$ denote $x_\la=a_1x_1+\dots+a_nx_n$. We have $\nu(\tau)=x_\la$ and it is now straightforward to deduce
\begin{thm}\label{NOmain}
$\Delta=\cQ_O(\la)-x_\la$.    
\end{thm}

\acknowledgements{We thank Michael Joswig for helpful comments.}

\printbibliography

\end{document}